\documentclass[11pt,bezier]{article}
\usepackage{amsmath,amssymb,amsfonts,color,euscript,graphicx}
\usepackage{color}

\newtheorem{prethm}{{\bf Theorem}}

\newenvironment{thm}{\begin{prethm}{\hspace{-0.5
               em}{\bf.}}}{\end{prethm}}

\newtheorem{prepro}[prethm]{{\bf Theorem}}

\newtheorem{preprop}[prethm]{{\bf Proposition}}

\newtheorem{precor}[prethm]{{\bf Corollary}}

\newenvironment{cor}{\begin{precor}{\hspace{-0.5
               em}{\bf.}}}{\end{precor}}

\newtheorem{preconj}[prethm]{{\bf Conjecture}}

\newtheorem{preremark}[prethm]{{\bf Remark}}

\newenvironment{remark}{\begin{preremark}\rm{\hspace{-0.5
               em}{\bf.}}}{\end{preremark}}

\newtheorem{preques}[prethm]{{\bf Question}}

\newenvironment{ques}{\begin{preques}\rm{\hspace{-0.5
               em}{\bf.}}}{\end{preques}}

               \newtheorem{predefn}[prethm]{{\bf Definition}}

\newtheorem{preexample}[prethm]{{\bf Example}}

\newenvironment{example}{\begin{preexample}\rm{\hspace{-0.5
               em}{\bf.}}}{\end{preexample}}

\newtheorem{prelem}[prethm]{{\bf Lemma}}

\newenvironment{lem}{\begin{prelem}{\hspace{-0.5
               em}{\bf.}}}{\end{prelem}}

\newtheorem{prelam}{{\bf Lemma}}

\newtheorem{preproof}{{\bf Proof.}}

\newenvironment{proof}[1]{\begin{preproof}{\rm
               #1}\hfill{$\Box$}}{\end{preproof}}

\newcommand{\diam}{\mathop{\mathrm{diam}}\nolimits}

\title{\bf \large  On the structure of the power graph and \\ the enhanced power graph 
of a group
\thanks
{ 
\newline \noindent 
2010{ \it Mathematics Subject Classification}: 05C25, 05C69, 20D60.}}

\author{ 
Ghodratollah Aalipour\thanks{Department of Mathematics and Computer Sciences, Kharazmi University, 50 Taleghani Avenue, Tehran, Iran and and Department of Mathematical and Statistical Sciences, University of Colorado Denver, CO 80217, USA (ghodrat.aalipour@ucdenver.edu)}  
\and Saieed  Akbari\thanks{Department of Mathematical Sciences, Sharif University of Technology (s\_akbari@sharif.edu)}
\and Peter J. Cameron\thanks{School of Mathematics and Statistics, University of St Andrews and School of Mathematical Sciences, Queen Mary, University of London (pjc20@st{\rm-}andrews.ac.uk)}
\and Reza Nikandish\thanks{Department of Basic Sciences, Jundi-Shapur University of Technology, Dezful, Iran. P.O. BOX 64615-334  (r.nikandish@ipm.ir)}
\and Farzad Shaveisi\thanks{”Department of Mathematics, Faculty of Sciences, Razi University, Kermanshah, Iran (f.shaveisi@razi.ac.ir)}
}
\date{}

\begin{document}
\maketitle
\begin{abstract}
{\small Let $G$ be a group. The  \emph{power graph} of $G$ is a graph with the vertex
set $G$, having an edge between two elements whenever one is a power of the other. We characterize  nilpotent groups whose power graphs have finite independence number. For a bounded exponent group, we prove its power graph is a perfect graph and we determine 
its clique/chromatic number. Furthermore, it is proved that for every group $G$, the clique number of the power graph of $G$ is at most countably infinite. We also measure how close the power graph is to the \emph{commuting graph} by introducing a new graph which lies in between. We call this new graph as the \emph{enhanced power graph}. For an arbitrary pair of these three graphs we characterize finite groups for which this pair of graphs are equal. }
\end{abstract}
\section{Introduction}

We begin with some standard definitions from graph theory and group theory.

Let $G$ be a graph with vertex set $V(G)$. If $x\in V(G)$, then the number of vertices adjacent to $x$ is called the \textit{degree} of $x$, and denoted by $\deg(x)$.
The \textit{distance} between two vertices in a graph is the number of
edges in a shortest path connecting them. The \textit{diameter} of
a connected graph $G$, denoted by $\diam (G)$, is the maximum distance
between any pair of vertices of $G$. If $G$ is disconnected, then $\diam (G)$ is defined to be infinite.
A \textit{star} is a graph in which there is a vertex adjacent to all other vertices, with no further edges. The \textit{center} of a star is a vertex that is adjacent to all other vertices.
Let $U\subseteq V(G)$. The induced subgraph on $U$ is denoted by $\langle U\rangle$. An \textit{independent set} is a set of vertices in a graph, no two of which are adjacent; that is, a set whose induced subgraph is null.
The {\it{independence number}} of a graph $G$ is the cardinality of the largest independent set and is denoted by $\alpha(G)$.
A subset $S$ of the vertex set of $G$ is called a \textit{dominating set} if for every vertex $v$ of $G$, either $v\in S$ or $v$ is adjacent to a vertex in $S$. The minimum size of dominating sets of $G$, denoted by $\gamma(G)$, is called the \textit{domination number} of $G$. A \textit{clique} in a graph is a set of pairwise adjacent vertices. The supremum of the sizes of cliques in $G$, denoted by
$\omega(G)$, is called the \textit{clique number} of $G$. By
$\chi(G)$, we mean the \textit{chromatic number} of $G$, i.e., the
minimum number of  colours which can be assigned to the vertices of
$G$ in such a way that every two adjacent vertices have different
 colours.

The cyclic group of order $n$ is denoted by $C_n$. A group $G$ is called \textit{periodic} if every element of $G$ has finite order. For every element $g\in G$, the order of $g$ is denoted by $o(g)$. 
If there exists an integer $n$ such that for all $g\in G$, $g^n=e$, where $e$ is the identity element of $G$, then $G$ is said to be of \textit{bounded exponent}. If $G$ is of bounded exponent, then the \textit{exponent} of $G$ is the least common multiple of the orders of its elements; that is, the least $n$ for which $g^n=e$ for all $g\in G$. A group $G$ is said to be \textit{torsion-free} if
apart from the identity every element of $G$ has infinite order. Let $p$  be a prime number. The \textit{$p$-quasicyclic group} (known also 
as the \textit{Pr\"{u}fer group}) is the $p$-primary component of $\mathbb{Q}/\mathbb{Z}$, that is, the unique maximal $p$-subgroup of $\mathbb{Q}/\mathbb{Z}$.
It is denoted by $ C_{p^\infty}$. The \textit{center} of a group $G$, denoted by $Z(G)$, is the set of elements that commute with every element of $G$. A group $G$ is called \textit{locally finite} if every finitely generated subgroup of $G$ is finite. A group is \textit{locally cyclic} if any finitely generated subgroup is cyclic.

Other concepts will be defined when needed.

Now, we define the object of interest to us in this paper.

Let $G$ be a group. The \textit{power graph} of $G$, denoted by $\mathcal{G}(G)$, is the graph whose vertex set is $G$, two elements being
adjacent if one is a power of the other. This graph was first introduced for semigroups in \cite{ghosh} and then was studied in \cite{cameron} and \cite{cameronghosh} for groups. It was shown that, for a finite group, the undirected
power graph determines the directed power graph up to isomorphism.
As a consequence, two finite groups which have isomorphic undirected
power graphs have the same number of elements of each order. The authors in \cite{cameronghosh} have shown that the only finite
group whose automorphism group is the same as that of its power graph is the Klein group
of order $4$. 

Our results about the power graph fall into four classes.
\begin{itemize}\itemsep0pt
\item In Section~\ref{independence}, we consider the independence number $\alpha(\mathcal{G}(G))$. We
show that if the independence number is finite then $G$ is a locally finite
group whose centre has finite index. Using this we are able to give precise
characterizations of nilpotent groups $G$ for which $\alpha(\mathcal{G}(G))$
is finite -- such a group (if infinite) is the direct product of a
$p$-quasicyclic group and a nilpotent $p'$-group.
\item In Section~\ref{colouring}, we show that the power graph of every group has clique number at most countable. A group
with finite clique number must be of bounded exponent. Hence we obtain a
structure theorem for abelian groups with this property, as well as showing
that it passes to subgroups and supergroups of finite index.
\item We do not know whether the chromatic number of every group is at most
countable; in Section~\ref{perfectness}, we prove this for periodic groups and for free groups. We show
that, if $G$ has bounded exponent, then $\mathcal{G}(G)$ is perfect.
\item Finally, in Section~\ref{misc} there are some miscellaneous results. A group is periodic if and only if
its power graph is connected, and in this case its diameter must be at most $2$.
Also we show that, if all vertex degrees in $\mathcal{G}(G)$ are finite, then
$G$ is finite.
\end{itemize}
In the recent paper \cite{feng-xuanlong-wang}, the authors prove that the power graph of every finite group is perfect. We acknowledge that our result on the perfectness along with all results in the Section \ref{sec-power-graph} were proved independently in 2011.

Another well-studied graph associated to a group $G$ is the \emph{commuting graph} of $G$. This graph appears to be first studied by Brauer and Fowler in 1955 in \cite{Brauer-Fowler} as a part of classification of finite simple groups. As the elements of the centre are adjacent to all other vertices, usually the vertices are assumed to be non-central. 
For more information on the commuting graph, see \cite{Araujo-et, Giudici-Pope, Woodcock} and the references therein. 

In Section \ref{power-commuting} we relate the power graph to the commuting graph and characterize when they are equal for finite groups. A new graph pops up while considering these graphs, a graph whose vertex set consists of all group elements, in which two vertices $x$ and $y$ are adjacent if they generate a cyclic group. We call this graph 
as  the \emph{enhanced power graph} of $G$ and we denote it by $\mathcal{G}_e(G)$. The enhanced power graph contains the power graph  and is a subgraph of the commuting graph. We further study some properties of this graph in the Section \ref{power-commuting}. 

We characterize the finite groups for which equality holds for any two of
these three graphs, and the solvable groups for which the power graph is equal
to the commuting graph. Other results are as follows:
\begin{itemize}\itemsep0pt
\item If the power graphs of $G$ and $H$ are isomorphic, then their enhanced
power graphs are isomorphic.
\item A maximal clique in the enhanced power graph is either a cyclic or a
locally cyclic subgroup.
\item $\mathcal{G}_e(G)$ has finite clique number if and only if $G$ has
finite exponent; if this holds, then the clique number of $\mathcal{G}_e(G)$
is equal to the largest order of an element of $G$. Also, for any group $G$,
the clique number of $\mathcal{G}_e(G)$ is at most countable.
\end{itemize}

\section{Power graphs of groups}\label{sec-power-graph}

\subsection{Independent sets in power graphs}\label{independence}

In this section we provide some results on the finiteness of the independence number of the power graphs. In the proof of the our first theorem, we need the following definition.
Let $G$ be a group and associate with $G$ a graph $\Gamma(G)$ as follows: the
vertices of $\Gamma(G)$ are the elements of $G$ and two vertices $g$ and $h$ of $\Gamma(G)$ are joined by an edge if and only if $g$ and $h$ do not commute, see \cite{akbari} and \cite{neuman} for more details. Now, we have the following result.
\begin{thm}\label{finiteindependent}
Let $G$ be a group and $\alpha({\mathcal {G}}(G))<\infty$. Then
\begin{enumerate}
\item[\rm (i)] $[G:Z(G)]<\infty$.
\item[\rm (ii)] $G$ is locally finite.
\end{enumerate}
\end{thm}
\begin{proof}
{(i) First we note that if $x$ and $y$ are adjacent in $\Gamma(G)$, then $x$ and $y$ are not adjacent in $\mathcal{G}(G)$. Thus $\omega(\Gamma(G))\leq\alpha({\mathcal {G}}(G))<\infty$. Hence \cite[Theorem 6]{neuman} implies that $[G:Z(G)]<\infty$.\\
(ii) Let $H$ be a finitely generated subgroup of $G$. Then by (i) and \cite[1.6.11]{robinson}, $Z(H)$ is finitely generated, too. So by the fundamental theorem for finitely generated abelian groups we find that $Z(H)\cong \mathbb{Z}^n\times C_{q_1}\times\cdots \times C_{q_k}$, where $n$ and $k$ are non-negative integers and every $q_i$, $1\leq i\leq k$, is a power of a prime number. Since $\alpha({\mathcal {G}}(\mathbb{Z}))=\infty$, we deduce that $H$ is a finite group and so the proof is complete.}
\end{proof}

Now, we characterize those abelian groups whose power graphs have finite independence number. First we need the following theorem.
\begin{thm}{\rm(\cite[4.3.11]{robinson})}\label{rob4.3.11}
If $G$ is an abelian group which is not torsion-free, then it has a non-trivial
direct summand which is either cyclic or quasicyclic.
\end{thm}
\begin{thm}\label{alphafinite}
Let $G$ be an abelian group such that $\alpha(\mathcal{G}(G))<\infty$. Then either $G$ is finite or $G\cong C_{p^\infty}\times H$, where $H$ is a finite group and $p\nmid |H|$.
\end{thm}
\begin{proof}
{If $G$ is torsion-free, then  $G$ contains $\mathbb{Z}$ and so $\alpha(\mathcal{G}(G))\geq\alpha(\mathcal{G}({\mathbb{Z}}))=\infty$, a contradiction. Thus by Theorem \ref{rob4.3.11}, $G=G_1\oplus H_1$, where $G_1$ is either cyclic or quasicyclic. If $H_1$ is trivial, then we are done. Otherwise, $\alpha(\mathcal{G}(H_1))<\infty$ implies that $H_1=G_2\oplus H_2$, where $G_2$ is either cyclic or quasicyclic. So $G=G_1\oplus G_2\oplus H_2$. By repeating this procedure and using $\alpha(\mathcal{G}(G))<\infty$, we deduce that there exists a positive integer $n$ such that $G\cong \bigoplus_{i=1}^n G_i$, where every $G_i$ is either cyclic or quasicyclic. We show that at most one $G_i$ is quasicyclic. By the contrary, suppose that $G$ contains the group $ C_{p^\infty}\times C_{q^\infty}$. It is not hard to see that for every positive integer $n$, $I_n=\{(1/p^i+\mathbb{Z},1/q^{n-i+1}+\mathbb{Z})\,:\, 1\leq i \leq n\}$ is an independent set of size $n$, a contradiction. So either $G\cong C_{p^\infty}\times \prod_{i=1}^n C_{p_i^{\alpha_i}}$ or $G\cong\prod_{i=1}^n C_{p_i^{\alpha_i}}$, where $p$ and $p_i$ are prime numbers. Now, suppose that the first case occurs. To complete the proof, we show that $p\neq p_i$, for every $i$, $1\leq i\leq n$. By contrary, suppose that $p=p_i$, for some $i$. Then $ C_{p^\infty}\times C_{p}$ is a subgroup of $G$. Since $ C_{p^\infty}\times\{1\}$ is an independent set in $\mathcal{G}( C_{p^\infty}\times  C_{p})$, we get a contradiction. So, the proof is complete.}
\end{proof}

\begin{thm}\label{alphapgroup}
Let $p$ be a prime number and $G$ be a $p$-group such that $\alpha(\mathcal{G}(G))<\infty$. Then either $G$ is finite or $G\cong C_{p^\infty}$.
\end{thm}
\begin{proof}
{Since $\alpha(\mathcal{G}(G))<\infty$, we deduce that $\alpha(\mathcal{G}(Z(G)))<\infty$. Thus by Theorem \ref{alphafinite}, either $Z(G)$ is finite or $Z(G)\cong  C_{p^\infty}$, for some prime number $p$. If $Z(G)$ is finite, then by Theorem \ref{finiteindependent}, $G$ is finite. Now, suppose that $Z(G)\cong  C_{p^\infty}$. To complete the proof, we show that $G$ is abelian. To the contrary, suppose that there exists $a\in G\setminus Z(G)$. Let $H=\langle Z(G)\cup\{a\}\rangle$. Clearly, $H$ is an abelian $p$-subgroup of $G$ and $\alpha(\mathcal{G}(H))<\infty$. So, by Theorem \ref{alphafinite}, $H\cong C_{p^\infty}\cong Z(G)$. Since every proper subgroup of $ C_{p^\infty}$ is finite, we get a contradiction. Hence $G$ is abelian and $G=Z(G)\cong C_{p^\infty}$.}
\end{proof}
Now, we exploit Theorem \ref{alphapgroup} to extend Theorem \ref{alphafinite} to nilpotent groups. 
\begin{remark}\label{decomposition}
Let $H$ and $K$ be two subgroups of $G$. If $H\cap K=\{e\}$, $G=HK$ and $H\subseteq Z(G)$, then $G\cong H\times K.$
\end{remark}

\begin{thm}\label{nilpotentalpha}
Let $G$ be an infinite nilpotent group. Then $\alpha(\mathcal{G}(G))<\infty$ if and only if $G\cong   C_{p^\infty}\times H$, for some prime number $p$, where $H$ is a finite group and $p\nmid |H|$.
\end{thm}
\begin{proof}
{First suppose that $G\cong   C_{p^\infty}\times H$, where $H$ is a finite group and $p\nmid |H|$.  Suppose to the contrary, $\{(s_n/p^{\alpha_n}+\mathbb{Z},g_n)\,:\, n\geq 1, s_n\in\mathbb{Z}, p\nmid s_n\,{\rm and}\,g_n\in H \}$ is an infinite independent
set of $\mathcal{G}(G)$. Since $H$ is a finite group, there exists $g\in H$ such that the infinite set
$$\{(s_n/p^{\alpha_n}+\mathbb{Z},g)\,:\, n\geq 1, s_n\in\mathbb{Z}\ {\rm and}\ p\nmid s_n \}$$
forms an independent set. Since $o(g)<\infty$, there exist $\alpha_i$ and $\alpha_j$ such that $p^{\alpha_i}\equiv p^{\alpha_j}({\rm mod
}\, o(g))$ and $\alpha_i>\alpha_j$. On the other hand, we know that $\gcd(s_i,p)=1$. So, let $t_i$ be the multiplicative inverse of $s_i$ in $ C_{p^{\alpha_j}}$. Thus by Chinese Reminder Theorem, there exists a positive integer $x$ such that $x\equiv t_is_j({\rm mod}\, p^{\alpha_j})$ and $x\equiv p^{\alpha_j-\alpha_i}({\rm mod}\,o(g))$. Therefore, we have
$$p^{\alpha_i-\alpha_j}x\frac{s_i}{p^{\alpha_i}}-\frac{s_j}{p^{\alpha_j}}=\frac{s_ix-s_j}{p^{\alpha_j}}\in \mathbb{Z},\hspace{.5cm} g^{xp^{\alpha_i-\alpha_j}}=g.$$
Thus, $(s_i/p^{\alpha_i}+\mathbb{Z},g)$  and $(s_j/p^{\alpha_j}+\mathbb{Z},g)$ are adjacent, a contradiction.

Conversely, suppose that $\alpha(\mathcal{G}(G))<\infty$. Then by Theorem \ref{finiteindependent}, $[G:Z(G)]<\infty$ and so $G=Z(G)H$, where $H$ is a finitely generated subgroup of $G$. Now, Theorem \ref{finiteindependent} implies that $H$ is finite. By Theorem \ref{alphafinite}, $Z(G)=AB$, where $A\cong C_{p^\infty}$ and $B$ is a finite group such that $p\nmid |B|$. Also, since $H$ is nilpotent, we have $H\cong H_pH_{p_1}\cdots H_{p_t}$, where $H_p$ and $H_{p_i}$ ($1\leq i\leq t$) are sylow $p$-subgroup and sylow $p_i$-subgroup of $H$, respectively. We show that $H_p\subseteq A$. To the contrary, suppose that $x\in H_p\setminus A$. Then $\langle  A,x\rangle$ is a $p$-group and so by Theorem \ref{alphapgroup},  $\langle A,x\rangle\cong  C_{p^\infty}\cong \langle A\rangle$. Since every proper subgroup of $ C_{p^\infty}$ is finite, we get a contradiction. Thus $H_p\subseteq Z(G)$ and so $G=ABH_{p_1}\cdots H_{p_t}$. Since $G$ is nilpotent, $BH_{p_1}\cdots H_{p_t}$ is a finite subgroup of $G$ and $p\nmid |BH_{p_1}\cdots H_{p_t}|$. Hence by Remark \ref{decomposition}, $G\cong A\times BH_{p_1}\cdots H_{p_t}$, as desired.~}
\end{proof}

\subsection{The colouring of power graphs}\label{colouring}

Let $G$ be a group. In this section, we first show that the chromatic number of the power graph of $G$ is finite if and only if the clique number of the power graph of $G$ is finite and this statement is also equivalent to that the exponent of $G$ is finite. Then it is proved that the clique number of the power graph of $G$ is at most countable. Finally, it is shown that the power graph of every bounded exponent group is perfect.

\begin{lem}\label{chrombond}
Let $G$ be a group. If $\omega(\mathcal{G}(G))$ is finite, then $G$ is of bounded exponent. 
\end{lem}
\begin{proof} {By the contrary, suppose that $G$ is not of bounded exponent. Then for every positive integer $k$, there is an element $g_k\in G$ such that $o(g_k)>2^k$. So one can easily show that $\{g_k^{2^i}\,|\,0\leq i\leq k\}$ is a clique of size $k+1$ in ${\mathcal {G}}(G)$. This implies that $\omega({\mathcal {G}}(G))=\infty$, a contradiction. The proof is complete. 
}
\end{proof}

\begin{remark} The proof uses the Axiom of Choice for families of finite sets. 
\end{remark}

\begin{cor}
Let $G$ be an abelian group and $\omega({\mathcal{G}}(G))<\infty$. Then there are some positive integers $r$ and $n_i$ and sets $I_i$, $1\leq i\leq r$ such that $$G\cong \prod_{i=1}^r\prod _{I_i} C_{n_i}.$$
\end{cor}
\begin{proof}
{Since $\omega({\mathcal{G}}(G))<\infty$, by Lemma \ref{chrombond}, $G$ is  bounded exponent. So the assertion follows from Pr\"{u}fer-Baer Theorem (see \cite[4.3.5]{robinson}).}
\end{proof}

\begin{thm}\label{conuntclique}
The clique number of the power graph of any group is at most countably infinite.
\end{thm}

\begin{proof} {Let $C$ be a clique in the power graph of $G$, and take
$x\in C$. Then the remaining vertices $y$ of $C$ are of two types:
\begin{itemize}
\item $y=x^n$ for some $n$;
\item $x=y^n$ for some $n$.
\end{itemize}
Clearly, there are at most countably many of the first type. We denote the set of  vertices of the second type by $C(n)$. We show that $C(n)$ is at most countably infinite. If there is only one $y$ in $C(n)$, then there is nothing to prove; so suppose there
are at least two elements in $C(n)$. We claim that every element in $C(n)$ has finite order. Choose $y,y'\in C(n)$. With no loss of generality, one can assume that $y'=y^k$, for some positive integer $k$. So $y^{(k-1)n}=1$. This implies that the orders of both $y$ and $y'$ are finite. Thus the claim is proved. Now, for every positive integer $k$, define $C(n,k)=\{y\in C(n)\,|\,o(y)=k\}$. By the claim, $C(n)=\bigcup_{k\geq1}C(n,k)$.  It is not hard to show that for every $a,b\in C(n,k)$, $\langle a\rangle=\langle b\rangle$ and so $C(n,k)$ is finite. Therefore, $C(n)$ is at most countably infinite.}
\end{proof}

We wonder if the same result holds for the chromatic number:  Does the power graph of every group have a countable chromatic number? 
A group is called a {\it pcc-group} if its power
graph has at most countable chromatic number. Free groups have this property by the next theorem. 
\begin{thm}\label{torsionpcc2} Every free group is a pcc-group.
\end{thm}
\begin{proof}
{By \cite[Corollary, p.51]{takashi}, in a free group, every non-identity element lies in a unique maximal
cyclic subgroup, generated by an element which is not a proper power
(when written as a reduced word). So, the power graph of a free group consists of
many copies of the power graph of an infinite cyclic group, with the
identity in all these copies identified.~}
\end{proof}
Next, we show that every abelian group is a pcc-group. First, we need the following result. 
\begin{lem}\label{torsionpcc} Every periodic group is a pcc-group.
\end{lem}
\begin{proof}
{Suppose that $G$ is a periodic group. For every positive integer $n$, let $G_n$ be the set of all elements of $G$ of order $n$. 
If $g,h\in G_n$ and $g$ and $h$ are adjacent in the power graph,
then $g$ and $h$ generate the same cyclic group. Hence, the induced subgraph on $G_n$ is a disjoint union of cliques of size $\phi(n)$.  
So one can colour the induced subgraph on $G_n$
with $\phi(n)$ colours. Clearly, $G=\bigcup_{n\geq 1} G_n$ and so the chromatic number of $\mathcal{G}(G)$ is at most countable.}
\end{proof}
Now, we prove that the class of pcc-groups contains the class of abelian groups. 

\begin{thm}
Every abelian group is a pcc-group. 
\end{thm}
\begin{proof}
{Let $G$ be abelian. Then $G$ can be embedded in a divisible abelian group (\cite[Theorem 4.1.6, p.98]{robinson}). It is known that  every divisible abelian group is of the form $H\times K$, where $H$ is a periodic group and $K$ is a direct sum of many copies of $\mathbb{Q}$, see \cite[Theorem 4.1.5, p.97]{robinson}. By  Lemma \ref{torsionpcc}, $H$ is a pcc-group. Now, the next two claims prove the assertion of the theorem.  

\subparagraph*{ Claim 1. } If $M$ and $N$ are two pcc-groups, then $M\times N$ is a pcc-group. 

Proof of Claim 1: Let  $f$ and $g$ be the proper countable colouring $\mathcal{G}(M)$ and $\mathcal{G}(N)$, respectively. Then the new map 
$\phi:M\times N\longrightarrow Image(f)\times Image(g)$ defined by $\phi(x,y)=(f(x),g(y))$ is a proper countable 
colouring for $\mathcal{G}(M\times N)$. This completes the proof of Claim 1.

\subparagraph*{Claim 2. } Every torsion-free abelian group is a pcc-group. 

Proof of Claim 2: Let $A$ be a torsion-free abelian group. We show that the connected components of $\mathcal{G}(A)$ are at most countable.  First, we show that the degree of any vertex is at most countable. Now, the identity element $0$ is an isolated vertex. Suppose that $x \neq 0$. There are only countably many multiples of $x$. Also, for each natural number $n$, there is at most one element $y$  such that $ny=x$. For, if $ny=nz=x$, then $n(y-z)=0$, so $y-z=0$ since $A$ is torsion-free. So there  are at most countably many elements of which $x$ is a multiple. Hence, the neighbourhood of $x$ is countable. This implies that the set of vertices of distance $2$ from $x$ is countable as well. By induction, we conclude that the set of vertices of distance $k$ from $x$ is countable. So, the component of $\mathcal{G}(A)$ which contains $x$ is countable. Now, we can colour each connected component with countably many colours. The proof of Claim 2 is complete.  
}
\end{proof}

Other classes that could be looked at would include solvable 
groups. One could also ask whether the class of pcc-groups is
extension-closed.

\subsubsection{Perfectness of the power graph}\label{perfectness}

A graph $G$ is called  \textit{perfect} if for every finite induced subgraph $H$ of $G$, $\chi(H)=\omega(H)$. The Strong Perfect Graph Theorem states that a finite graph $G$ is perfect if and only if neither $G$ nor $\overline{G}$ (the complement of $G$) contains an induced odd cycle of length at least $5$, see \cite[Theorem 14.18]{bondy}. However, this is a deep theorem, and we do not need it to
prove our results.

Utilizing Lemma \ref{chrombond} to colour the power graph with a finite set of  colours we require the group to be bounded exponent. Here we show that for such groups the resulting power graph is always perfect and can be finitely coloured. To prove this result we facilitate the concepts of comparability graph.

Let $\leq$ be a binary relation on the elements of a set $P$. If $\leq$ is reflexive and transitive, then $(P,\leq)$ is called a \textit{pre-ordered set}. All partially ordered sets are pre-ordered. The {\it comparability graph} of a pre-ordered set $(P,\leq)$ is the simple graph $\Upsilon(P)$ with the vertex set $P$ and two distinct vertices $x$ and  $y$ are adjacent if and only if either $x\leq y$  or $y\leq x$ (or both).

\begin{thm}\label{comparabilitygraph}
Let $m$ be a positive integer and $P$ be a pre-ordered set {\rm (}not necessarily finite{\rm)} whose maximum chain size is $m$. Then the comparability graph $\Upsilon(P)$ is perfect and $$\omega(\Upsilon(P))=\chi(\Upsilon(P))=m.$$
\end{thm}

\begin{proof}
{The result is well known for comparability graphs of partial orders; our
proof is a slight extension of this. Since the class of comparability graphs is closed under taking induced subgraphs, it
is enough to prove that the comparability graph of $P$ has equal clique number and
chromatic number. Clearly,  a clique in a $\Upsilon(P)$ is a chain in $P$, while
a colouring is a partition into antichains.

First we show that $\omega(\Upsilon(P))=m$. Let $C$ be a clique in $\Upsilon(P)$. 
Then $C$ is a chain in $P$, and so $|C|\leq m$. Thus $\omega(\Upsilon(P))=m$.

Now, we show that $\chi(\Upsilon(P))\leq m$. We form a directed graph by
putting an arc from $x$ to $y$ whenever $x\le y$ but $y\not\le x$; and, if $C$
is an equivalence class of the relation $\equiv$ defined by $x\equiv y$ if
$x\le y$ and $y\le x$, then take an arbitrary directed path on the elements
of $C$. Clearly,  the longest directed path contains $m$ vertices. Let $P_i$ be
the set of elements $x$ for which the longest directed path ending at $x$
contains $i$ vertices. It is easy to see that $P_i$ is an independent set; these
sets partition $P$ into $m$ classes.}
\end{proof}
Now, we show that the power graph of a group is the comparability graph of a pre-ordered set. First, we define some notations. 
Let $n$ be a positive integer and $\mathcal{D}(n)$ be the set of all divisors of $n$ in $\mathbb{N}$. Define a relation $\preceq$ on $\mathcal{D}(n)$ by $r\preceq s$ if and only if $r\mid s$. Clearly, $(\mathcal{D}(n),\preceq)$ is a partially ordered set. Denote the set of all chains of $(\mathcal{D}(n),\preceq)$ by ${\mathcal{C}}(n)$. Using this convention we are able to determine the clique/chromatic number of the power graph of a group of bounded exponent (see Lemma \ref{chrombond}).  
\begin{thm}\label{main-theorem}
Let $G$ be a group of exponent $n$. Then $\mathcal{G}(G)$ is a perfect graph and $$\chi(\mathcal{G}(G))=\omega(\mathcal{G}(G))=\max\Bigg\{\sum\limits_{\substack{d\in C\\ G_d\neq \emptyset}}\phi(d)\,:\, C\in {\mathcal{C}}(n)\Bigg\}\leq n, $$ where $G_d$ is the set of elements of $G$ of order $d$, for some  divisor $d$ of $n$.
\end{thm}

\begin{proof}
{First we consider two following facts:\\

{\bf{Fact 1}}. Suppose that $d$ is a divisor of $n$ and $G_d\neq \emptyset$. If $g,h\in G_d$ and $g$ and $h$ are adjacent in the power graph,
then $g$ and $h$ generate the same cyclic group. So $\langle G_d\rangle$ is a disjoint union of cliques of size $\phi(d)$. Therefore; if $x$ is an element of a clique
$H$ of $\langle G_d\rangle$ adjacent to an element $y$ of a clique $K$ of $\langle G_{d'}\rangle$, then every element of $H$ is adjacent to every element of $K$ and moreover, $d\mid d'$ or $d'\mid d$.\\

{\bf{Fact 2}}. Note that  if $z$ is an element of order $d$, then for each divisor $d'$ of $d$, $z^{\frac{d}{d'}}$ is of order $d'$. So for each clique $T$ of $\langle G_d\rangle$, every element of $T$ is adjacent to every element of a clique $S$ of $\langle G_{d'}\rangle$.\\

Since $\{G_d\,:\,  G\ {\rm has\ an\ element\ of\ order\ }d\}$ forms a partition for $G$, Fact $1$ implies that every maximal clique of $\mathcal{G}(G)$ has the form $Cl_1\cup\cdots \cup Cl_m$, where $Cl_i$ is a clique of $\langle G_{d_i}\rangle$ of size $\phi(d_i)$ and $\{d_1,\ldots,d_m\}$ is a chain of length $m$ belonging to ${\mathcal{C}}(n)$. Moreover; by Fact $2$, we deduce that for every chain $\{d_1,\ldots,d_m\}$ in ${\mathcal{C}}(n)$, there exists a clique for $\mathcal{G}(G)$ of this form. Now, by $|Cl_i|=\phi(d_i)$, we conclude
that $$\omega(\mathcal{G}(G))=\max\Bigg\{\sum\limits_{\substack{d\in C\\ G_d\neq \emptyset}}\phi(d)\,:\, C\in {\mathcal{C}}(n)\Bigg\}\leq \sum_{d\mid n}\phi(d)=n.$$

Define a pre-ordering $\leq$ on $G$ by $x\leq y$ if and only if $x$ is a power of $y$. Clearly, the power graph of $G$ is the comparability graph of $\leq$ and so by Theorem \ref{comparabilitygraph}, the power graph of $G$ is perfect. Thus $\chi(\mathcal{G}(G))=\omega(\mathcal{G}(G))$ and the proof is complete.}
\end{proof}

The next two corollaries are direct consequences of Lemma \ref{chrombond} and Theorem \ref{main-theorem}. 

\begin{cor}
For every group $G$, the following statements are equivalent:
\begin{enumerate}
\item[\rm (i)] $\chi(\mathcal{G}(G))<\infty$;
\item[\rm(ii)] $\omega(\mathcal{G}(G))<\infty$;
\item[\rm(iii)] $G$ is bounded exponent.
\end{enumerate}
Moreover, the chromatic number of $\mathcal{G}(G)$ does not
exceed the exponent of $G$.
\end{cor}
\begin{cor}
Let $G$ be an abelian group of exponent $n$. Then $$\chi(\mathcal{G}(G))=\omega(\mathcal{G}(G))=\max\left\{\sum_{d\in C}\phi(d)\,:\, C\in {\mathcal{C}}(n)\right\}.$$
\end{cor}
\begin{cor}\label{characteromega}
Let $H$ be a subgroup of $G$ and $[G:H]<\infty$. Then $\omega({\mathcal{G}}(H))<\infty$ if and only if $\omega({\mathcal{G}}(G))<\infty$.
\end{cor}

The following example shows that a similar assertion does not hold for the independence number.
\begin{example}
Let $G= C_2\times C_{2^\infty}$ and $H=\{0\}\times  C_{2^\infty}$. Thus $[G:H]=2$. Since ${\mathcal{G}}(H)$ is a complete graph, $\alpha(H)=1$.  Clearly, the set $\{1\}\times C_{2^\infty}$ is independent and so $\alpha(G)=\infty$.
\end{example}

\subsection{Miscellaneous properties}\label{misc}

We conclude this section with three miscellaneous properties of the power graph of a group.

\begin{thm}\label{bipartite}
If ${\mathcal {G}}(G)$ is a triangle-free graph, then $G$ is isomorphic to a direct product of $ C_2$ and ${\mathcal {G}}(G)$ is a star.
\end{thm}
\begin{proof}
{First we show that the order of every element of $G$ is at most 2. Let $a\in G$. 
If $o(a)\geq 3$, then $\{e,a,a^2\}$ is a triangle, a contradiction. So $G$ is an elementary abelian 2-group. Therefore, by Pr\"{u}fer-Baer Theorem, $G$ is isomorphic to a direct product of $ C_2$ and so ${\mathcal {G}}(G)$ is a star with the center $e$.}
\end{proof}

The following theorem characterizes those groups whose power graphs are connected.
\begin{thm}\label{periodicdiameter}
Let $G$ be a group. The following statements are equivalent.
\begin{enumerate}
 \item[\rm(i)] ${\mathcal {G}}(G)$ is connected;
\item[\rm(ii)] $G$ is periodic;
\item[\rm(iii)] $\gamma({\mathcal {G}}(G))=1$;
\item[\rm(iv)] $\diam ({\mathcal{G}}(G))\leq 2$.
\end{enumerate}
\end{thm}
\begin{proof}
{(i)$\Longrightarrow$(ii) Let $x$ ($x\neq e$)  be a vertex of ${\mathcal {G}}(G)$. We show that $x$ is of finite order in $G$. Since ${\mathcal {G}}(G)$ is connected, there is a path from $x$ to $e$. Let $y$ be the adjacent vertex to $e$ in this path. So the order of $y$ is finite. Now, suppose that $t$ is the adjacent vertex to $y$ in this path. Then the order of $t$ is finite, too. By repeating this procedure, we deduce that the order of $x$ is finite. So $G$ is periodic.\\
(ii)$\Longrightarrow$(iii) Since every element in $G$ has a finite order, $\{e\}$ is a dominating set.\\
The parts (iii)$\Longrightarrow$(iv) and (iv)$\Longrightarrow$(i) are clear.}
\end{proof}

\begin{thm}
If $\deg(g)<\infty$, for every $g\in G$, then $G$ is a finite group.
\end{thm}
\begin{proof}
{Let $g\in G$. Since $\deg(g)<\infty$, $g$ has a finite order in $G$. Thus $G$ is a periodic group and so $e$ is adjacent to every other vertices of $\mathcal{G}(G)$. Since $\deg(e)<\infty$, we deduce that $G$ is finite.}
\end{proof}

\begin{remark}
Let $G\cong \prod_{i\geq 1} C_2$. Then $\mathcal{G}(G)$ is an infinite star with the center $0$. This shows that in the previous theorem the condition $\deg(g)<\infty$, for every $g\in G$ is necessary.
\end{remark}

\section{Power graph and commuting graph}\label{power-commuting}

Let $G$ be a group. If the vertices $x$ and $y$ are joined in the power graph
of $G$, then they are joined in the commuting graph; so the power graph is a
spanning subgraph of the commuting graph.

\begin{ques}\label{power=commuting}
For which groups is it the case that the power graph is
equal to the commuting graph?
\end{ques}

The identity is joined to all others in the commuting graph; so if the two
graphs are equal, then $G$ is a periodic group. 

\begin{thm}\label{powergraph-equals-commuting}
Let $G$ be a finite group with power graph equal to commuting graph. Then
one of the following holds:
\begin{itemize}
\item $G$ is a cyclic $p$-group;
\item $G$ is a semidirect product of $C_{p^a}$ by $C_{q^b}$, where $p$ and
$q$ are primes with $a,b>0$, $q^b\mid p-1$ and $C_{q^b}$ acts faithfully on
$C_{p^a}$;
\item $G$ is a generalized quaternion group.
\end{itemize}
\end{thm}

\begin{proof}
{Let $G$ have power graph equal to commuting graph; that is, if two elements
commute, then one is a power of the other. Then $G$ contains no subgroup
isomorphic to $C_p\times C_q$, where $p$ and $q$ are primes, since this group
fails the condition.

A theorem of Burnside~\cite[Theorem 12.5.2]{hall} says that a $p$-group containing
no $C_p\times C_p$ subgroup is cyclic or generalized quaternion. So all Sylow
subgroups of $G$ are of one of these two types.

Suppose that all Sylow subgroups are cyclic. Then $G$ is metacyclic
\cite[Theorem 9.4.3]{hall}.
The cyclic normal subgroup of $G$ has order divisible by one prime only,
say $p$. Its centraliser in $G$ is a Sylow $p$-subgroup $P$ of $G$, since it
contains no elements of order coprime to $p$. Hence $G$ is a semidirect
product of $P$ and a cyclic group $Q$ of order coprime to $p$, necessarily a
cyclic $q$-group for some prime $q$. If $|Q|=q^b$, then we have $q^b\mid p-1$.

So we may suppose that $G$ has generalized
quaternion Sylow $2$-subgroups. By Glauberman's Z*-Theorem~\cite{glauberman},
$G/O(G)$ has a central involution, where $O(G)$ is the maximal normal subgroup of $G$ of odd order. 
This involution must act fixed-point-freely on $O(G)$, so $O(G)$ is abelian,
and hence cyclic of prime power order. But a generalized quaternion group
cannot act faithfully on such a group. So $O(G)=1$. Then the involution in
$G$ is central, so $G$ is a $2$-group, necessarily a
generalized quaternion group.}
\end{proof}

\begin{remark}
In the infinite case, there are other examples, such as \emph{Tarski monsters},
which are infinite groups whose non-trivial proper subgroups are all cyclic
of prime order $p$. Probably no classification is possible.
\end{remark}

In the next theorem, we extend Theorem \ref{powergraph-equals-commuting} to solvable groups. 

\begin{thm}
Let $G$ be a solvable group with power graph equal to commuting graph. Then
one of the following holds:
\begin{itemize}
\item $G$ is a cyclic $p$-group;
\item $G$ is a semidirect product of $C_{p^a}$ by $C_{q^b}$, where $p$ and
$q$ are primes with $a,b>0$, $q^b\mid p-1$ and $C_{q^b}$ acts faithfully on
$C_{p^a}$;
\item $G$ is a generalized quaternion group; 
\item $G$ is the $p$-quasicyclic group $C_{p^\infty}$; 
\item $G$ is a semidirect product of  $p$-quasicyclic group $C_{p^\infty}$ and a finite cyclic group.
\end{itemize}
\end{thm}

\begin{proof}
{
We use this fact that every finitely generated periodic solvable group is finite. We know that $G$ is periodic.
We  show that there are no three elements whose order are distinct primes. Suppose that $o(a)=p$, $o(b)=q$ and $o(c)=r$, where $p$, $q$ and $r$ are distinct primes. Let $H$ be the subgroup generated by $a$, $b$ and $c$. Then $H$ is finite. Clearly, the power graph and the commuting graph of $H$ are equal. This contradicts Theorem \ref{powergraph-equals-commuting}. Thus the order of every finite subgroup of $G$ is $p^\alpha q^\beta$, for some non-negative
integers $\alpha$ and $\beta$. By the second part of Theorem \ref{powergraph-equals-commuting}, we may assume that $\beta$ is bounded, because $q^\beta | p-1$. Also, by Theorem \ref{powergraph-equals-commuting}, the order of every element of $G$ is a $p$-power or a $q$-power, because every cyclic subgroup is a $p$-group or a $q$-group. 
If $\alpha, \beta >0$, then by the second part of Theorem \ref{powergraph-equals-commuting},  $\langle a, b \rangle $  is semidirect product of $\langle a \rangle$ and $\langle b \rangle$. So, $\langle a \rangle$ and
$\langle b \rangle$ are both cyclic  (even in the case $q=2$). Let $N$ be the set of all elements of $G$ whose orders are $p$-power.
We show that $N$ is an abelian normal subgroup of $G$. To see this first we show that if $x$ and $y$ are two elements of $N$, then $xy=yx$. Let $S$ be the subgroup
generated by $x$ and $y$. Then $S$ is a finite group of order $p^\alpha q^\beta$. If $\beta=0$, then by Theorem \ref{powergraph-equals-commuting}, $S$ is a cyclic
$p$-group and we are done. So assume that $\beta>0$. Let $N_1$ and $Q$ be Sylow $p$-subgroup and Sylow $q$-subgroup of $S$, respectively.
Then by Theorem \ref{powergraph-equals-commuting} both are cyclic. Now, by \cite[Theorem 6.2.11]{scott}, $Q$ has a normal complement. So, $N_1\lhd S$. This implies that $x,y\in N_1$ and
so $xy=yx$. Thus we conclude that $N$ is an abelian $p$-subgroup of $G$. Now, by the definition of $N$, $N\lhd G$.

Now, let $Q$ be a $q$-subgroup of $G$ which has maximum size. We prove that $G=NQ$. 

Let $a\in G$ be an element of $G$ whose order is
$q$-power. Let $M$ be the subgroup generated by $a$ and $Q$. Then $M=P_1Q_1$, where $P_1$ and $Q_1$ are Sylow $p$-subgroup and Sylow $q$-subgroup of $M$, respectively and $P_1\lhd M$. But $Q_1$ is a conjugate of $Q$ and so $M=P_1Q$. Since $a\in M$, we have $a=bc$, where $b\in P_1$ and $c\in Q$.
But $P_1\subseteq P$ and this implies that $a\in PQ$, as desired. So $G=NQ$.

Since $N$ is an abelian group, the commuting graph of $N$ and so the power graph of $N$ is a complete graph.
Now, Theorem \ref{alphapgroup} yields that $N$ is finite or $N=Z_{p^\infty}$. }
\end{proof}
\section{The enhanced power graph}

\subsection{Definition and properties}

In the Section \ref{sec-power-graph} we investigated some properties of power graphs of groups. In Theorem \ref{powergraph-equals-commuting}, we  characterized finite groups for which the power graph is the same as the commuting graph. Now, it is natural to ask if they are not equal, how close these graphs are. To tackle this problem we introduce an intermediate graph. This graph can be regarded as a measurement for this difference. Given a group $G$, the \emph{enhanced power graph} of $G$ denoted by $\mathcal{G}_e(G)$ is the graph with vertex set $G$, in which $x$ and $y$ are joined if and only if there exists
an element $z$ such that both $x$ and $y$ are powers of $z$. 

The power graph and commuting graph behave well on restriction to a subgroup
(that is, if $H\le G$, then the power graph of $H$ is the induced subgraph
of the power graph of $G$ on the set $H$, and similarly for the commuting
graph). Because of the existential quantifier in the definition, it is not
obvious that the same holds for the enhanced graph. That this is so is a
consequence of the fact that $x$ and $y$ are joined in the enhanced power
graph if and only if $\langle x,y\rangle$ is cyclic. Note that
\begin{itemize}
\item the power graph is a spanning subgraph of the enhanced power graph;
\item the enhanced power graph is a spanning subgraph of the commuting graph.
\end{itemize}

In the next remark, we use the concept of graph squares. For a graph $H$, the \textit{square} of $H$ is a graph with the same vertex set as $H$ in which two vertices are adjacent if their distance in $H$ is at most two.  
\begin{remark}
 If  we assume that the (undirected or directed) power graph has a loop at each vertex, then the enhanced power graph lies between the power graph and
its square. We already saw that it contains the power graph (as a spanning subgraph).
Now, let $x$ and $y$ be two vertices joined by a path $(x,z,y)$ of length
$2$ in the power graph. There are four cases in the directed power graph
$D=\vec{\mathcal{G}}(G)$:
\begin{itemize}\itemsep0pt
\item $(x,z),(z,y)\in E(D)$. Then $x$ is a power of $z$, and $z$ a power of
$y$; so $x$ is a power of $y$, and $(x,y)\in E(D)$.
\item $(z,x),(y,z)\in E(D)$. Dual to the first case.
\item $(x,z),(y,z)\in E(D)$. Then $x$ and $y$ are powers of $z$, so they
are joined in the square of the power graph.
\item $(z,x),(z,y)\in E(D)$. In this case there is nothing we can say.
\end{itemize}
\end{remark}

Also the following holds:

\begin{thm}
Let $G$ and $H$ be finite groups. If the power graphs of $G$ and $H$ are
isomorphic, then their enhanced power graphs are also isomorphic.
\end{thm}

\begin{proof}
{Note that $x$ and $y$ are joined in the enhanced power graph
if and only if there is a vertex $z$ which dominates both in the directed
power graph. So the theorem follows from the main theorem of~\cite{cameron}.}
\end{proof}
\subsection{Comparing to the power graph and commuting graph}

\begin{ques}
For which (finite) groups is the power graph equal to the enhanced power graph?
\end{ques}

This question connects with another graph associated with a finite group,
the \emph{prime graph}, defined by Gruenberg and Kegel~\cite{gk}: the
vertices of the prime graph of $G$ are the prime divisors of $|G|$, and
vertices $p$ and $q$ are joined if and only $G$ contains an element of 
order $pq$. To state the next result we need a definition. The group $G$ is a
\emph{$2$-Frobenius group} if it has normal subgroups $F_1$ and $F_2$
such that $F_1<F_2$, $F_2$ is a Frobenius group with Frobenius kernel $F_1$,
and $G/F_1$ is a Frobenius group with Frobenius kernel $F_2/F_1$.

In the statement of the following theorem, $p$ and $q$ denote distinct primes.

\begin{thm}\label{power=enhance}
For a finite group $G$, the following conditions are equivalent:
\begin{itemize}
\item[(a)] the power graph of $G$ is equal to the enhanced power graph;
\item[(b)] every cyclic subgroup of $G$ has prime power order;
\item[(c)] the prime graph of $G$ is a null graph.
\end{itemize}
A group $G$ with these properties is one of the following: a $p$-group; a
Frobenius group whose kernel is a $p$-group and complement a $q$-group;
a $2$-Frobenius group where $F_1$ and $G/F_2$ are $p$-groups and
$F_2/F_1$ is a $q$-group; or $G$ has a normal $2$-subgroup with quotient
group $H$, where $S\le H\le\mathrm{Aut}(S)$ and $S\cong A_5$ or $A_6$.
\end{thm}

All these types of group exist. Examples include $S_3$ and $A_4$ (Frobenius
groups); $S_4$ (a $2$-Frobenius group); $A_5$, $A_6$ and $2^4:A_5$.

\begin{proof}
{Let $p$ and $q$ be distinct primes. The cyclic group of order $pq$ does
not have property (a); so a group satisfying (a) must also satisfy (b).
Conversely, suppose that (b) holds. If $x$ and $y$ are adjacent in the
enhanced power graph, then $\langle x,y\rangle$ is cyclic, necessarily of
prime power order; so it must be generated by one of $x$ and $y$, and
so $x$ and $y$ are adjacent in the power graph.

Clearly,  (b) and (c) are equivalent.

Now, let $G$ be a group satisfying these conditions. Either $G$ is a $p$-group
for some prime $p$, or the prime graph of $G$ is disconnected.
Now, we use the result of Gruenberg and Kegel~\cite{gk} (stated and proved in
Williams~\cite{williams}), asserting that a finite group with disconnected
prime graph  is Frobenius or $2$-Frobenius, simple, $\pi_1$ by simple, simple
by $\pi_1$-solvable, or $\pi_1$ by simple by $\pi_1$. Here $\pi_1$ is the
set of primes in the connected component of the prime graph containing $2$,
assuming that $|G|$ is even; and a $2$-Frobenius group is a group $G$ with
normal subgroups $F_1<F_2$ such that $F_2$ is a Frobenius group with kernel
$F_1$, and $G/F_1$ is a Frobenius group with kernel $F_2/F_1$.

It follows from the work of Frobenius that a Frobenius complement either has
all Sylow subgroups cyclic (and so is metacyclic) or has $\mathrm{SL}(2,3)$ or
$\mathrm{SL}(2,5)$ as a normal subgroup. These last two cases cannot occur,
since the central involution commutes with elements of order $3$. In the 
first case, the results
of Gruenberg and Kegel (see the first corollary in Williams~\cite{williams})
show that the Frobenius complement has only one prime divisor.

In the case of a $2$-Frobenius group, an element of the Frobenius complement
in the top group centralises some element of $F_1$; so $F_1$ and $G/F_2$ must
be $p$-groups for the same prime $p$.

In the remaining cases, it can be read off from the tables in
Williams~\cite{williams} that the simple group can only be $A_5$ or $A_6$,
and the conclusions of the theorem follow since $\pi_1=\{2\}$.}
\end{proof}
\begin{ques}
For which (finite) groups is the enhanced power graph
equal to the commuting graph?\label{q:qq}
\end{ques}
Again, we have a lot of information about such a group. 
\begin{thm}
For a finite group $G$, the following conditions are equivalent:
\begin{itemize}
\item[(a)] the enhanced power graph of $G$ is equal to its commuting graph;
\item[(b)] $G$ has no subgroup $C_p\times C_p$ for $p$ prime;
\item[(c)] the Sylow subgroups of $G$ are cyclic or (for $p=2$) generalized
quaternion.
\end{itemize}
A group satisfying these conditions is either a cyclic $p$-group for some
prime $p$, or satisfies the following: if $O(G)$ denotes the largest normal
subgroup of $G$ of odd order, then $O(G)$ is metacyclic,
$H=G/O(G)$ is a group with a unique involution $z$, and $H/\langle z\rangle$
is a cyclic or dihedral $2$-group, a subgroup of
$\mathrm{P}\Gamma\mathrm{L}(2,q)$ containing $\mathrm{PSL}(2,q)$ for $q$ an
odd prime power, or $A_7$.
\end{thm}

An example of a group for the second case is the direct product of
the Frobenius group of order $253$ by $\mathrm{SL}(2,5)$.

\begin{proof}
{The group $C_p\times C_p$ has commuting graph not equal to its enhanced power
graph, so cannot be a subgroup of a group satisfying (a); thus (a) implies
(b). Conversely, suppose that (b) holds. Let $x$ and $y$ be elements of $G$
which are adjacent in the commuting graph. Then $\langle x,y\rangle$ is
abelian, and hence is the direct product of two cyclic groups, say
$C_r\times C_s$. Under hypothesis (b), we must have $\gcd(r,s)=1$, and
so $\langle x,y\rangle\cong C_{rs}$; thus $x$ and $y$ are joined in the
enhanced power graph.

Conditions (b) and (c) are equivalent by a theorem of Burnside
\cite[Theorem 12.5.2]{hall}.

Suppose that a Sylow $2$-subgroup $P$ of $G$ is cyclic or generalized
quaternion. If $P$
is cyclic, then by Burnside's transfer theorem \cite[Section 14.3]{hall},
$G$ has a normal $2$-complement: that is, if $O(G)$ is the largest normal
subgroup of $G$ of odd order, then $G/O(G)\cong P$. If $P$ is generalized
quaternion, then by Glauberman's $Z^*$-Theorem, $H=G/O(G)$ has a unique
central involution $z$. Put $Z=\langle z\rangle$. Then the Sylow
$2$-subgroups $Q$ of $H/Z$ are dihedral; the Gorenstein--Walter theorem
\cite{gw} shows that $H/Z$ is isomorphic to a subgroup of
$\mathrm{P}\Gamma\mathrm{L}(2,q)$ containing $\mathrm{PSL}(2,q)$ (for
odd $q$), or to the alternating group $A_7$, or to $Q$. For any such 
group $H^*=H/Z$, an argument of Glauberman (which can be found in \cite{bc})
shows that there is a unique double cover $H$ with a single involution.}
\end{proof}

\begin{ques}
What can be said about the difference of the enhanced power graph and the
power graph, or the difference of the commuting graph and the enhanced power
graph? In particular, for which groups is either of these graphs connected?
\end{ques}


\subsection{Maximal cliques in the enhanced power graph}
We will now look at maximal cliques in the enhanced power graph. This requires
a lemma which looks trivial, but we couldn't find an easier proof of it than the one below.

\begin{lem}\label{2-cyclic-3-cyclic}
Let $x,y,z$ be elements of a group $G$, and suppose that $\langle x,y\rangle$,
$\langle x,z\rangle$ and $\langle y,z\rangle$ are cyclic. Then
$\langle x,y,z\rangle$ is cyclic.
\end{lem}

\begin{proof}
{
The result clearly holds if one of $x,y,z$ is the
identity; so suppose not. Now, a cyclic group cannot contain elements of
both finite and infinite order, so either all three elements have finite
order, or all three have infinite order.
\subparagraph{Case 1:} $x,y,z$ have finite order. Then they generate a finite
abelian group $A$.

We first note that it suffices to do the case where the orders of $x,y,z$ are
powers of a prime $p$. For $A$ is the direct sum of $p$-groups for various
primes $p$; each $p$-group is generated by certain powers of $x,y,z$; and if
each $p$-group is cyclic, then so is $A$. With this assumption, suppose that $A$ is not cyclic. Since $\langle x,y\rangle$
is cyclic, $A$ is the sum of two cyclic groups, and so contains a subgroup
$Q\cong C_p\times C_p$. Each of $\langle x\rangle$, $\langle y\rangle$ and
$\langle z\rangle$ intersects $Q$ in a subgroup of order $p$; let these
subgroups be $X,Y,Z$. Since $\langle x,y\rangle$ is cyclic, it meets $Q$ in
a subgroup of order $p$; so $X=Y$. Similarly $X=Z$. So $\langle x,y,z\rangle$ meets
$Q$ in a subgroup of order $p$, contradicting the assumption that
$Q\le\langle x,y,z\rangle$.

\subparagraph{Case 2:} $x,y,z$ have infinite order. 

Then they generate a free
abelian group; since $\langle x,y\rangle$ is cyclic, we see that
$A=\langle x,y,z\rangle$ has rank at most $2$. Consider the $\mathbb{Q}$-vector space $A\otimes_\mathbb{Z}\mathbb{Q}$, which
has dimension at most $2$. Since $\langle x,y\rangle$ is cyclic, the
$1$-dimensional subspaces $\langle x\rangle\otimes_\mathbb{Z}\mathbb{Q}$ and
$\langle y\rangle\otimes_\mathbb{Z}\mathbb{Q}$ have non-empty intersection, and so are
equal. Similarly for $\langle z\rangle\otimes_\mathbb{Z}\mathbb{Q}$. Thus
$A\otimes_\mathbb{Z}\mathbb{Q}$ is $1$-dimensional, so $A$ is cyclic.
}
\end{proof}

Now, we have the following characterization of the maximal cliques in the enhanced power graph. 

\begin{lem} A maximal clique in the enhanced power graph is either
a cyclic subgroup or a locally cyclic subgroup.
\end{lem}

\begin{proof}
{
Clearly,  a cyclic or locally cyclic subgroup is a clique. Suppose that $C$ is a maximal clique. If $x,y\in C$, then by 
 Lemma \ref{2-cyclic-3-cyclic}, every element of $\langle x,y\rangle$ is joined to every element $z\in C$;
so $C\cup\langle x,y\rangle$ is a clique. By maximality of $C$, we have
$\langle x,y\rangle\subseteq C$; so $C$ is a subgroup. Now, a simple induction
shows that any finite subset of $C$ generates a cyclic group, so that
$C$ is locally cyclic.
}
\end{proof} 

\begin{remark}
Locally cyclic groups include the additive group of $\mathbb{Q}$ (or the
subgroup consisting of rationals whose denominators only involve primes from
a prescribed set), and direct sums of copies of the $p$-quasicyclic groups (the Pr\"ufer groups) 
$C_{p^\infty}$ for distinct primes $p$.
\end{remark}
Now, we have two immediate corollaries:
\begin{cor} Let $G$ be a group. Then $\omega(\mathcal{G}_e(G))<\infty$ if and only if $G$ is a group of finite exponent. 
If these conditions hold, then 
\[\omega(\mathcal{G}_e(G))=\max\{o(g)\,:\, g\in G\}.\]
\end{cor}

\begin{remark} Note that this may be smaller than the exponent
of $G$.
\end{remark}
\begin{proof}
{
Clearly,  if $G$ is not a bounded exponent group, then $\mathcal{G}(G)$  as a subgraph of $\mathcal{G}_e(G)$ has infinite clique number 
by Lemma \ref{chrombond}. Now, let $G$ be a periodic group. Then the subsets $G_n=\{g\in G: o(g)=n\}$, for $n\in\mathbb{N}$, partition $G$ into at most countably many subsets. On each of these subsets the power graph and the enhanced power graph coincide. 
In particular, if $G$ is bounded
exponent, then there are only finitely many classes. It is clear
that, if $x$ and $y$ have the same order and generate a cyclic group, then
each is a power of the other. 
}
\end{proof}

\begin{cor} A clique in the enhanced power graph of a group is at most
countable.
\end{cor}

\begin{proof}
{
For a  locally cyclic group is isomorphic to a subgroup of $\mathbb{Q}$ or $\mathbb{Q}/\mathbb{Z}$ and hence countable, see \cite[Exercise 5, p.105]{robinson}. 
}
\end{proof}


\section*{Open problems}
This paper concerns  several graph theoretical parameters of the power graph of a group. In Section \ref{independence} we studied groups whose power graph has a finite independence number. In Theorem \ref{nilpotentalpha}, we proved that 
if $G$ is a nilpotent group and $\alpha(\mathcal{G})<\infty$, then either $G$ is a finite group or 
$G\cong   C_{p^\infty}\times H$, for some prime number $p$, where $H$ is a finite group and $p\nmid |H|$. This result motivates us to pose the following question.

\begin{ques}
Let $G$ be an infinite group. Is it true that $\alpha(\mathcal{G}(G))<\infty$ if and only if $G\cong   C_{p^\infty}\times H$, where $H$ is a finite group and $p\nmid |H|$.
\end{ques}

In Section \ref{colouring}, we  showed that the chromatic number of the power graph of $G$ is finite if and only if the clique number of the power graph of $G$ is finite and this statement is also equivalent to the finiteness of the exponent of $G$. We proved that the clique number of the power graph of $G$ is at most countable. We also introduced the concept of pcc groups and we posed the following question.  

\begin{ques} Is it true that the chromatic number of the power graph
of any group is at most countably infinite?
\end{ques}

It might be interesting to ask how much of Lemma \ref{chrombond} can be proved without the Axiom of Choice. Is there any way of showing that the
chromatic number of a group of finite exponent is finite? A good test case for this question would be an abelian group of exponent~3. Colouring the non-identity elements with two colours requires choosing one of each pair
$\{x,x^{-1}\}$, which requires AC (as Bertrand Russell famously pointed out).

In the study of the commuting graph,
it is normal to delete vertices which lie in the centre of the group, since
they would be joined to all other vertices. Similarly, in the study of the generating graph of a $2$-generator
group, the identity is an isolated vertex and is usually excluded.  This convention is not used for
the power graph. So any problem we raise will have two different forms, 
depending on which convention we use. For the power graph, the question of whether to include or exclude the
identity is more interesting. Some of the results will be completely different in the two cases especially those dealing with
connectedness. For example, if the identity is excluded, Theorem \ref{periodicdiameter} fails,
since indeed the power graph of the infinite cyclic group is connected when
the identity is discarded. The next question seems interesting. 

\begin{ques}
Which groups do have the property that the power graph is connected when the identity is removed? 
\end{ques}
Or more generally:
\begin{ques}
Which groups do have the property that the power graph is connected when the set
of vertices which dominate the graph is removed? 
\end{ques}
\noindent 
The following question is the second version for Question \ref{power=commuting}. 
\begin{ques}
For which groups, are the induced subgraphs of the power graph and the commuting graph on
$G\setminus\{e\}$ are equal. Note
that free groups have this property.
\end{ques}

\begin{ques}
Consider the difference of the power graph and commuting graph, the graph in
which $x$ and $y$ are joined if they commute but neither is a power of the
other. What can be said about this difference graph? In particular, for which
groups is it connected? Again this question can be asked with or without
the identity. Note that in a periodic group the identity is isolated in the
difference graph, but this is not true for arbitrary infinite groups.
\end{ques}

{}

\end{document}